\documentclass[11pt]{article}

\newtheorem{theorem}{Theorem}[section]
\newtheorem{lemma}[theorem]{Lemma}
\newtheorem{proposition}[theorem]{Proposition}
\newtheorem{corollary}[theorem]{Corollary}
\newtheorem{remark}[theorem]{Remark}
\newtheorem{example}[theorem]{Example}
\newtheorem{definition}[theorem]{Definition}
\newenvironment{proof}{\par\noindent{\it Proof.}\ }{\hfill\rule{1.2ex}{1.2ex}\par}

\newcommand{\Hb}{\widehat H}
\newcommand{\Ht}{\widetilde H}
\newcommand{\bt}{\widetilde\beta}
\newcommand{\R}{{\bf R}}

\newcommand{\Ric}{\mathop{\rm Ric}}
\newcommand{\Diam}{\mathop{\rm Diam}}

\newcommand{\im}{\mathop{\rm Im}}

\newcommand{\vol}{\mathop{\rm vol}}

\title{CURVATURE, DIAMETER AND BOUNDED BETTI NUMBERS\\
{\large A REVISIT VIA THE GENERALIZED MARGULIS LEMMA}}
\author{Zhongmin Shen and Jyh-Yang Wu}
\date{}

\begin{document}
\maketitle

\begin{abstract}
We consider the image of the comparison homomorphism from bounded
cohomology to ordinary cohomology.  Its dimension is called the bounded
Betti number.  In our earlier work we proved that the total bounded Betti
number of a closed Riemannian $n$-manifold satisfying
$\Ric\geq -(n-1)$ and $\Diam\leq D$ is bounded by a constant depending
only on $n$ and $D$.  The original proof required the amenability of
the images of sufficiently small local fundamental groups, a statement
conjectured by Gromov.  A weaker Ricci curvature version was stated by
Cheeger and Colding in connection with their structure theory for
Ricci limit spaces.  Kapovitch and Wilking later gave a complete proof
of the generalized Margulis lemma under a lower Ricci curvature bound,
with the stronger conclusion that the relevant local groups contain
nilpotent subgroups of uniformly bounded index.  We show explicitly
how their theorem supplies the uniform amenability radius needed in
the original covering argument.  We also discuss the topological
meaning of the bounded part of ordinary cohomology and its relation to
the $\ell^1$-seminorm on homology.
\end{abstract}

\bigskip
\noindent
{\bf Key words:} Diameter, Ricci curvature, sectional curvature,
bounded cohomology and bounded Betti number.

\medskip
\noindent
{\bf 2000 Mathematics Subject Classification:} 53C21, 53C23.

\section{Introduction}

Let $M$ be a connected topological space.  The inclusion of the complex
of bounded singular cochains into the complex of all singular cochains
induces the comparison homomorphism
\[
        c^i:\Hb^i(M;\R)\longrightarrow H^i(M;\R).
\]
We denote its image by
\[
        \Ht^i(M):=\im c^i
\]
and put
\[
        \bt_i(M):=\dim \Ht^i(M).
\]
We called $\bt_i(M)$ the $i$th bounded Betti number.

The terminology should not obscure an important point.  The bounded
cohomology $\Hb^*(M;\R)$ is, under the usual hypotheses, determined
isometrically by the fundamental group, whereas its image
$\Ht^*(M)$ in ordinary cohomology need not be determined by the
fundamental group alone.  Thus $\Ht^i(M)$ records an interaction
between the non-amenable information carried by the fundamental group
and the ordinary $i$-dimensional topology of $M$.

Our main theorem was the following.

\begin{theorem}[Bounded Betti number theorem]\label{thm:A}
For every $n$ and $D>0$ there is a constant $C(n,D)$ such that every
closed Riemannian $n$-manifold $(M,g)$ satisfying
\[
        \Ric\geq -(n-1),\qquad \Diam(M)\leq D
\]
satisfies
\[
        \sum_{i=0}^n\bt_i(M)\leq C(n,D).
\]
\end{theorem}

This is different from the corresponding statement for ordinary Betti
numbers.  Under a lower Ricci curvature bound and an upper diameter
bound, the higher ordinary Betti numbers need not admit such a uniform
bound.  The bounded part therefore isolates a portion of ordinary
cohomology which is much more rigid under these weak geometric
hypotheses.

The local group-theoretic ingredient in the original argument was an
amenability statement conjectured by Gromov.  In the published paper
we referred to the Cheeger--Colding work for this input.  Their Ricci
limit theory provided the structural foundation and indicated the
needed conclusion, but the complete proof of the generalized Margulis
lemma, including a uniform index bound, was later given by Kapovitch
and Wilking.  We record their theorem below and show explicitly that it
implies the uniform amenability radius required in our proof.

\section{The bounded part of ordinary cohomology}

Let $C^i(M;\R)$ denote the singular cochains and let
$\widehat C^i(M;\R)$ be the subcomplex consisting of bounded cochains.
For $u\in\widehat C^i(M;\R)$ put
\[
       \|u\|_\infty=\sup_\sigma |u(\sigma)|,
\]
where the supremum is over all singular $i$-simplices.  The inclusion
$\widehat C^*(M;\R)\subset C^*(M;\R)$ induces
\[
       c^i:\Hb^i(M;\R)\longrightarrow H^i(M;\R).
\]

\begin{definition}
The bounded part of $H^i(M;\R)$ is
\[
       \Ht^i(M)=\im\{c^i:\Hb^i(M;\R)\to H^i(M;\R)\}.
\]
Its dimension
\[
       \bt_i(M)=\dim\Ht^i(M)
\]
is the bounded Betti number.
\end{definition}

For $\alpha\in\Ht^i(M)$ one may use the quotient seminorm
\[
 \|\alpha\|_\infty
   =\inf\{\|\widehat\alpha\|_b:
        \widehat\alpha\in\Hb^i(M;\R),\
        c^i(\widehat\alpha)=\alpha\}.
\]

The following observation gives a useful interpretation of
$\Ht^i(M)$.  For a singular chain
$z=\sum_j a_j\sigma_j$, put
\[
        \|z\|_1=\sum_j|a_j|,
\]
and let the same notation denote the induced $\ell^1$-seminorm on
homology.

\begin{proposition}\label{prop:duality}
If $\alpha\in\Ht^i(M)$ and $z\in H_i(M;\R)$, then
\[
       |\langle\alpha,z\rangle|
       \leq \|\alpha\|_\infty\,\|z\|_1.
\]
In particular, every bounded ordinary cohomology class vanishes on
every homology class whose $\ell^1$-seminorm is zero.
\end{proposition}

\noindent
Indeed, choose a bounded representative of $\alpha$ and evaluate it on
singular cycles approaching the $\ell^1$-seminorm of $z$.

Thus it is useful to introduce
\[
       N_i(M)=\{z\in H_i(M;\R):\|z\|_1=0\}.
\]
Every element of $\Ht^i(M)$ induces a linear functional on
\[
       H_i(M;\R)/N_i(M).
\]

The standard duality principle for the $\ell^1$-seminorm and bounded
cohomology gives a more precise statement.  The Kronecker pairing
induces an isometric isomorphism
\[
 \left(\Ht^i(M),\|\cdot\|_\infty\right)
 \simeq
 \left(H_i(M;\R)/N_i(M),\|\cdot\|_1\right)^*.
\]
Here the quotient by $N_i(M)$ changes the $\ell^1$-seminorm into a
norm.  Consequently,
\[
 \bt_i(M)
 =\dim\left(H_i(M;\R)/N_i(M)\right).
\]
Indeed, Proposition~\ref{prop:duality} gives one direction.  The
reverse direction follows from the duality principle, together with
the Hahn--Banach theorem applied to the finite dimensional quotient
$H_i(M;\R)/N_i(M)$.

This suggests the following interpretation: the bounded part is the
portion of ordinary cohomology which detects homology through bounded
cochains, or equivalently through the $\ell^1$-geometry of singular
cycles.  It is therefore not simply a measure of the size of
$\pi_1(M)$.  Rather, it measures how bounded-cohomological information
coming from the fundamental group survives in the ordinary topology
of $M$.

\begin{remark}
This viewpoint also explains why $\Ht^i(M)$ can behave very differently
from both $\Hb^i(M)$ and $H^i(M)$.  The comparison map can have a large
kernel and a small image, or the ordinary cohomology group itself can
be too small to receive the bounded classes.
\end{remark}

\section{Examples}

\begin{example}[Amenable fundamental group]
If $\pi_1(M)$ is amenable, then
\[
        \Hb^i(M;\R)=0\qquad (i\geq1)
\]
in the standard setting of bounded cohomology.  Consequently
\[
        \Ht^i(M)=0,\qquad \bt_i(M)=0\quad (i\geq1).
\]
In particular, tori and closed manifolds with virtually nilpotent
fundamental group have trivial bounded part in positive degrees.
\end{example}

\begin{example}[Closed negatively curved manifolds]
For a closed negatively curved manifold, the comparison map is
surjective in degrees $i\geq2$ by the classical bounded-cohomological
results used in the original paper.  Thus
\[
        \Ht^i(M)=H^i(M;\R),\qquad
        \bt_i(M)=\beta_i(M),\quad i\geq2.
\]
This represents the opposite extreme from the amenable case.
\end{example}

\begin{example}[Hyperbolic homology $3$-sphere]
Let $M$ be a closed hyperbolic integral homology $3$-sphere.  Its
bounded cohomology can be nontrivial in degree two, whereas
\[
        H^2(M;\R)=0.
\]
Therefore
\[
        \Ht^2(M)=0
\]
although $\Hb^2(M;\R)$ is nontrivial.  This example makes clear that
the bounded part is not determined by the size of bounded cohomology
alone.
\end{example}

\begin{example}[Products with an amenable factor]
Let $N$ be a closed manifold and consider $M=N\times S^1$.
The circle contributes an amenable factor to the fundamental group.
The mapping theorem for bounded cohomology shows that amenable normal
subgroups do not contribute new bounded cohomology.  Nevertheless,
ordinary cohomology acquires the additional classes supplied by the
$S^1$ factor.  Hence the passage from $\Hb^*(M)$ to $\Ht^*(M)$ again
depends on how the bounded classes map into the enlarged ordinary
cohomology.
\end{example}

\section{The generalized Margulis lemma under a Ricci lower bound}

The following theorem of Kapovitch and Wilking settles the generalized
Margulis lemma conjectured by Gromov and gives precisely the local
group-theoretic information needed here.

\begin{theorem}[Kapovitch--Wilking generalized Margulis lemma]
\label{thm:KW}
For each positive integer $n$, there are constants $\epsilon(n)>0$
and $C(n)>0$ with the following property.  Let $(M^n,g)$ be a complete
Riemannian manifold and suppose that
\[
        \Ric_M>-(n-1)
\]
on the metric ball $B_1(p)$.  Then the image of the natural
homomorphism
\[
 \pi_1(B_{\epsilon(n)}(p),p)
       \longrightarrow \pi_1(B_1(p),p)
\]
contains a nilpotent subgroup $N$ of index at most $C(n)$.
Moreover, $N$ admits a nilpotent basis of length at most $n$.
\end{theorem}

In particular, the image above is virtually nilpotent and hence
amenable.

\begin{remark}[Historical comment]
For a lower sectional curvature bound, related local fundamental group
results were established by Fukaya and Yamaguchi.  A weaker Ricci
curvature version was stated by Cheeger and Colding in connection with
their structure theory for Ricci limit spaces.  Kapovitch and Wilking
gave a complete generalized Margulis lemma under a lower Ricci
curvature bound and, importantly, obtained a uniform bound depending
only on the dimension for the index of the nilpotent subgroup.
\end{remark}

For the application to bounded Betti numbers we only require
amenability, so the full strength of the index bound in
Theorem~\ref{thm:KW} is not needed.  Nevertheless, the virtually
nilpotent conclusion gives a considerably sharper description of the
local fundamental groups than the amenability statement used in the
original argument.

The following corollary is the exact form needed below.  It also makes
clear why the local target $\pi_1(B_1(p),p)$ in
Theorem~\ref{thm:KW} is sufficient for a statement in $\pi_1(M,p)$.

\begin{corollary}[Uniform amenability radius]\label{cor:amenrad}
For every positive integer $n$ there is $r(n)>0$ such that, if $M^n$
is complete and
\[
             \Ric_M\geq -(n-1),
\]
then, for every $p\in M$, the image of
\[
 \pi_1(B_{r(n)}(p),p)\longrightarrow\pi_1(M,p)
\]
is virtually nilpotent, and hence amenable.
\end{corollary}

\begin{proof}
First suppose that the strict inequality
$\Ric_M>-(n-1)$ holds on $B_1(p)$.  Put $r(n)=\epsilon(n)$, where
$\epsilon(n)$ is supplied by Theorem~\ref{thm:KW}.  Let $G_p$ be the
image of
\[
 \pi_1(B_{r(n)}(p),p)\longrightarrow\pi_1(B_1(p),p).
\]
Theorem~\ref{thm:KW} says that $G_p$ is virtually nilpotent.  The image
of $\pi_1(B_{r(n)}(p),p)$ in $\pi_1(M,p)$ is the image of $G_p$ under
the homomorphism induced by $B_1(p)\longrightarrow M$.  It is therefore
a quotient of $G_p$.  A quotient of a virtually nilpotent group is
virtually nilpotent, and every virtually nilpotent group is amenable.

For the non-strict lower bound, replace $g$ temporarily by
$g_2=4g$.  Then
\[
 \Ric_{g_2}\geq-{n-1\over 4}g_2>-(n-1)g_2.
\]
Applying the strict case to $g_2$ shows that the image of the
$g$-ball of radius $\epsilon(n)/2$ in $\pi_1(M,p)$ is virtually
nilpotent.  Thus $r(n)=\epsilon(n)/2$ works uniformly.  Notice that
this rescaling does not change any fundamental group or inclusion
homomorphism.
\end{proof}

\section{Amenable coverings and bounded Betti numbers}

An open subset $U\subset M$ will be called amenable in $M$ if, for
every $x\in U$, the image of
\[
       \pi_1(U,x)\longrightarrow\pi_1(M,x)
\]
is amenable.

The key topological input is the amenable covering theorem of
Gromov--Ivanov.

\begin{lemma}[Amenable covering factorization]\label{lem:Ivanov}
Let ${\cal U}=\{U_j\}$ be an amenable open covering of $M$ and let
$N({\cal U})$ be its nerve.  Then the comparison map
\[
       \Hb^*(M;\R)\longrightarrow H^*(M;\R)
\]
factors through $H^*(|N({\cal U})|;\R)$.
Consequently,
\[
       \bt_i(M)\leq \dim H^i(|N({\cal U})|;\R).
\]
\end{lemma}

This lemma is the bridge between the local fundamental group and the
bounded part of global cohomology.

\section{Revised proof of the bounded Betti number theorem}

We now repeat the original covering argument, with
Corollary~\ref{cor:amenrad} replacing the conjectural input.

\par\noindent{\it Proof of Theorem~\ref{thm:A}.}\ 
Fix $n$ and $D$, and let $r=r(n)$ be the radius in
Corollary~\ref{cor:amenrad}.  Decreasing $r$ if necessary, assume that
$r\leq1$.  For every $p\in M$, the image of
$\pi_1(B_r(p),p)$ in $\pi_1(M,p)$ is virtually nilpotent and hence
amenable.

Choose a maximal collection of pairwise disjoint balls
\[
        B(p_j,r/2),\qquad j=1,\ldots,m.
\]
Then the balls $B(p_j,r)$ cover $M$.  Bishop--Gromov relative volume
comparison gives, for each $j$,
\[
 {\vol B(p_j,r/2)\over\vol M}
 \geq {V_{-1}(r/2)\over V_{-1}(D+1)},
\]
where $V_{-1}(a)$ denotes the volume of an $a$-ball in the simply
connected space of constant curvature $-1$.  Since the small balls are
pairwise disjoint, this gives a uniform bound
\[
        m\leq C_1(n,D).
\]
More precisely,
\[
 m\leq
 \frac{\displaystyle\int_0^{D+1}(\sinh t)^{\,n-1}\,dt}
      {\displaystyle\int_0^{r/2}(\sinh t)^{\,n-1}\,dt}.
\]

Each member of the covering is amenable in $M$.  Let $N$ be its nerve.
Since the covering contains at most $C_1(n,D)$ sets, the total number
of simplices of $N$ is bounded by a constant depending only on $n$ and
$D$.  Hence
\[
        \sum_i\dim H^i(|N|;\R)\leq C_2(n,D).
\]
By Lemma~\ref{lem:Ivanov}, the comparison map factors through
$H^*(|N|;\R)$, and therefore
\[
        \sum_i\bt_i(M)
        \leq\sum_i\dim H^i(|N|;\R)
        \leq C_2(n,D).
\]
This proves the theorem.
\hfill\rule{1.2ex}{1.2ex}\par

\subsection*{A geometric reformulation}

The preceding proof controls more than an abstract image of the
comparison map.  Recall that
\[
 N_i(M)=\{z\in H_i(M;\R):\|z\|_1=0\}.
\]
By the duality discussion in Section~2,
\[
 \bt_i(M)
 =\dim\left(H_i(M;\R)/N_i(M)\right).
\]
Theorem~\ref{thm:A} may therefore be stated equivalently as follows.

\begin{corollary}[Geometric reformulation]\label{cor:geometric}
For every positive integer $n$ and every $D>0$, there is a constant
$C=C(n,D)$ such that every closed Riemannian $n$-manifold satisfying
\[
       \Ric(M)\geq-(n-1),\qquad \Diam(M)\leq D,
\]
satisfies
\[
 \sum_{i=0}^n
 \dim\left(
 H_i(M;\R)/\{z\in H_i(M;\R):\|z\|_1=0\}
 \right)
 \leq C(n,D).
\]
\end{corollary}

Thus a lower Ricci curvature bound and an upper diameter bound control
the number of independent homological directions having nonzero
$\ell^1$-seminorm.  Ordinary homology may still be complicated, but
the quotient obtained by discarding all zero-$\ell^1$ directions has
uniformly bounded total dimension.  This gives a direct topological
and geometric interpretation of the bounded Betti numbers.

\begin{remark}
The passage from the local Kapovitch--Wilking group to the group used
by the amenable covering theorem is exactly the composition
\[
 \pi_1(B_r(p),p)\longrightarrow\pi_1(B_1(p),p)
 \longrightarrow\pi_1(M,p).
\]
The second image is a quotient of the first image.  This quotient
observation, rather than a subgroup argument, is what transfers virtual
nilpotence and amenability to $M$.
\end{remark}

\section{Gromov--Hausdorff collapse}

The same viewpoint clarifies the vanishing theorem from the original
paper.

\begin{theorem}\label{thm:collapse}
Let $X$ be a compact metric space and let $n$ be fixed.  There is
$\delta=\delta(n,X)>0$ such that, for a closed Riemannian
$n$-manifold satisfying
\[
        \Ric\geq-(n-1)
\]
and
\[
        d_{GH}(M,X)<\delta,
\]
one has
\[
        \bt_i(M)=0\qquad\quad\hbox{for }i>\dim X.
\]
\end{theorem}

The proof constructs an amenable covering of $M$ whose nerve has
dimension at most $\dim X$.  The generalized Margulis lemma provides
the required amenability of sufficiently small pieces.  Lemma
\ref{lem:Ivanov} then implies that the comparison map factors through
the cohomology of this low-dimensional nerve.

This formulation suggests a useful interpretation.  During collapse,
ordinary cohomology may contain classes in dimensions above the
dimension of the limit space, but those classes cannot lie in the
bounded part once the collapse is sufficiently strong.  Thus bounded
ordinary cohomology detects a portion of topology which cannot remain
hidden entirely in the collapsing directions.

\bigskip
\noindent
Zhongmin Shen\\
Department of Mathematical Sciences\\
Indiana University Indianapolis\\
402 N. Blackford Street,\\
Indianapolis, IN 46202-3216, U.S.A.\\
\\
Email: zshen@iu.edu

\bigskip
\noindent
Jyh-Yang Wu\\
Department of Information Management\\
Yuan Ze University\\
No. 135, Yuandong Road\\
Zhongli District, Taoyuan City 32003, Taiwan\\
Email: jywu@saturn.yzu.edu.tw

\end{document}